\documentclass[preprint]{imsart}

\RequirePackage[OT1]{fontenc}
\RequirePackage{amsthm,amsmath}
\RequirePackage[numbers]{natbib}
\RequirePackage[colorlinks,citecolor=blue,urlcolor=blue]{hyperref}

\arxiv{1454756 }

\startlocaldefs
\numberwithin{equation}{section}
\theoremstyle{plain}
\newtheorem{thm}{Theorem}[section]
\newtheorem{lem}{Lemma}[section]
\newtheorem{defn}{Definition}[section]
\newtheorem{ex}{Example}[section]
\newtheorem{rem}{Remark}[section]

\newtheorem{cor}{Corollary}[section]

\endlocaldefs

\begin{document}

\begin{frontmatter}
\title{Calculation of Lebesgue Integrals by Using Uniformly  Distributed Sequences}
\runtitle{Calculation of Lebesgue Integrals }

\begin{aug}
\author{\fnms{Gogi} \snm{Pantsulaia}\thanksref{t1}\ead[label=e1]{g.pantsulaia@gtu.ge}}
\and
\author{\fnms{Tengiz } \snm{Kiria}\ead[label=e2]{t.kiria@gtu.ge}}

\affiliation{}

\thankstext{t1}{The research for this paper was partially supported by Shota
Rustaveli National Science Foundation's Grant no FR/116/5-100/14}
\runauthor{G.Pantsulaia and T.Kiria}

\address{I.Vekua Institute of Applied Mathematics, Tbilisi - 0143, Georgian Republic \\
\printead{e1}\\
Georgian Technical University, Tbilisi - 0175, Georgian Republic\\
  \printead*{e2}
}
\end{aug}

\begin{abstract} We present modified  proof of a certain version of Kolmogorov's
strong law of large numbers for calculation of Lebesgue Integrals by using uniformly  distributed sequences in $(0,1)$. We  extend the result of C. Baxa and J. Schoi$\beta$engeier (cf.\cite{BaxSch2002}, Theorem 1, p. 271) to a maximal  set of uniformly distributed (in $(0,1)$) sequences $S_f \subset(0,1)^{\infty}$ which  strictly contains  the set of sequences  of the form $(\{\alpha n\})_{n \in
{\bf N}}$ with irrational number $\alpha$ and  for which $\ell_1^{\infty}(S_f)=1$,  where $\ell_1^{\infty}$ denotes the infinite power of the linear Lebesgue measure  $\ell_1$ in $(0,1)$.
\end{abstract}



\end{frontmatter}

\section{Introduction}
In this note we show that the technique for numerical calculation of some one-dimensional Lebesgue  integrals is similar to the technique
which  was given  by Hermann Weyl's \cite{Weyl} celebrated theorem as follows.

\medskip

\begin{thm}(\cite{KuNi74}, Theorem 1.1, p. 2) The sequence $(x_n)_{n \in N}$ of real numbers is
u.d. mod 1 if and only if for every real-valued continuous function $f$ defined
on the closed unit interval $[0, 1]$ we have
\begin{align*}
~~~~~~~~~~~~~~~~~~\lim_{ N \to \infty}\frac{\sum_{n=1}^Nf(\{x_n\})}{N} =\int_{0}^{1} f(x) dx,   ~~~~~~~~~~~~~~~~~~~~~~~~ (1.1)
\end{align*}
where $\{\cdot\}$ denotes the fractional part of the real number.
\end{thm}

Main corollaries of this theorem successfully were used in Diophantine approximations and have
applications to Monte-Carlo integration (see, for example, \cite{KuNi74},\cite{Hardy2}, \cite{Hardy1}). During the last decades the methods of
the theory of uniform distribution modulo one  have been intensively used for calculation of improper Riemann integrals(see, for example,
\cite{Sobol73}, \cite{BaxSch2002}).

In this note we are going to consider some applications of  Kolmogorov  strong law of large numbers which can be considered as a certain extension of the  Hermann Weyl's above mentioned theorem from the class of Riemann's integrable functions to the class of  Lebesgue integrable functions. We present  our  proof of this century theorem which differs from Kolmogorov's original proof. Further,   by using this theorem we  present a certain improvement of  the following result of C. Baxa and  J. Schoi$\beta$engeier

 \begin{thm}(\cite{BaxSch2002}, Theorem 1, p. 271)Let $\alpha$ ne an irrational number,  $\mathbf{Q}$ be a set of all rational numbers  and $F \subseteq [0,1]\cap \mathbf{Q}$ be finite. Let
$f : [0,1] \to R$ be an integrable, continuous almost everywhere and locally bounded on $[0,1] \setminus F$. Assume further that for every $\beta \in F$
there is some neighbourhood $U$ of $\beta$ such that   $f$ is either bounded or monotone in $[0,\beta) \cap U$ and in $(\beta,1]\cap U$ as well.
Then the following conditions are equivalent:

1) $\lim_{n \to \infty}\frac{f(x_n)}{n}=0$;

2) ~$\lim_{N \to \infty}\frac{1}{N}\sum_{k=1}^Nf(x_k) $ exists;

3) $\lim_{N \to \infty}\frac{1}{N}\sum_{k=1}^Nf(x_k) = \int_{(0,1)}f(x)dx$;

are equivalent
 \end{thm}

  More precisely, we  will extend the result of Theorem 1.2 to a maximal set $S_f \subset (0,1)^{\infty}$ of uniformly distributed (in $(0,1)$)sequences strictly containing all sequences  of the form $(\{\alpha n\})_{n \in
{\bf N}}$ where $\alpha$ is an irrational numbers and  for which $\ell_1^{\infty}(S_f)=1$,  where $\ell_1^{\infty}$ denotes the infinite power of the linear Lebesgue measure  $\ell_1$ in $(0,1)$.

The paper is organized as follows.

In Section 2  we consider some auxiliary notions and facts from the theory of uniformly distributed sequences and probability theory. In
Section 3 we present our main results.

\section{Auxiliary notions and facts}

\begin{defn} A  sequence $s_1, s_2, s_3, \cdots$ of real numbers from the interval $[0,1]$ is said
to be  uniformly distributed in the interval $[0,
1]$ if for any subinterval $[c, d]$ of the $[0, 1]$ we have
$$
\lim_{n \to \infty}\frac{\#(\{s_1, s_2, s_3, \cdots, s_n\} \cap
[c,d])}{n}=d-c,
$$
where $\#$ denotes the counting measure.
\end{defn}


\begin{ex} (\cite{KuNi74}, Exercise 1.12, p. 16) The sequence of all
multiples of an irrational $\alpha$
 \begin{align*}0, \{\alpha\}, \{2\alpha\}, \{3\alpha\}
\cdots
\end{align*}  is uniformly distributed in $(0,1)$, where $\{\cdot\}$ denotes the fractional part of the real number.
\end{ex}

\begin{lem}(\cite{KuNi74} Theorem 2.2, p.183)
Let $S$ be a set of all elements of
$[0,1]^{\infty}$ which are uniformly
distributed in the interval $[0,1]$. Then
$\ell_1^{\infty}(S)=1$, where $\ell_1^{\infty}$  denotes the infinite power of the standard linear Lebesgue measure $\ell_1$  in
$[0,1]$.
\end{lem}

\begin{lem} (Toeplitz Lemma (\cite{Shiryaev1980}, Lemma 1, p. 377) ) Let $(a_n)_{n \in {\bf N}}$ be a sequence of non-negative numbers,
$b_n=\sum_{i=1}^na_i, b_n>0$ for each $n \ge 1$ and $b_n \uparrow \infty,$  when $n \to \infty$. Let $(x_n)_{n \in {\bf N}}$ be a sequence of
real numbers such that $ \lim_{n \to \infty}x_n=x$. Then
$$\lim_{n \to \infty}\frac{1}{b_n}\sum_{j=1}^na_jx_j=x.$$
In particular, if $a_n=1$ for $n \in {\bf N}$, then
$$
\lim_{n \to \infty}\frac{1}{n}\sum_{k=1}^nx_k=x.
$$
\end{lem}

\begin{lem}
 Let $f$ be a Lebesgue integrable  non-negative  real-valued function on $(0,1)$.
  Then the following inequality
    $$
  \int_{(0,1)^{\infty}}(\frac{1}{N}(\sum_{k=1}^N\big(f(x_k) \chi_{ \{x :  f(x)<k \epsilon \} }(x_k) -\int_{ \{x :  f(x)<k \epsilon \}}f(x)dx\big)^2 d
  \ell_1^{\infty}((x_i)_{i \in N}) \le
  $$
  $$
   2 \epsilon \int_0^1f(x)dx
  $$
 holds true.
 \end{lem}

\begin{proof} We have
 $$
  \int_{(0,1)^{\infty}}(\frac{1}{N}\sum_{k=1}^N\big(f(x_k) \chi_{ \{x :  f(x)<k \epsilon \} }(x_k) -\int_{ \{x :  f(x)<k \epsilon \}}f(x)dx\big))^2 d
  \ell_1^{\infty}((x_i)_{i \in N}) \le
  $$

$$
  \frac{1}{N^2} \sum_{k=1}^N \int_0^1f^2(x)\chi_{ \{x :  f(x)<k \epsilon \}}(x)dx \le
  $$

$$
\sum_{k=1}^N \frac{1}{k^2} \int_0^1f^2(x)\chi_{ \{x :  f(x)<k \epsilon \}}(x)dx \le
$$

$$
\sum_{k=1}^{\infty}\frac{1}{k^2} \int_0^1f^2(x)\chi_{ \{x :  f(x)<k \epsilon \}}(x)dx =
$$

$$
\sum_{n=1}^{\infty} \frac{1}{n^2} \sum_{k=1}^n\int_0^1f^2(x)\chi_{ \{ (k-1)\epsilon \le f(x)<k \epsilon \}}(x)dx=
$$

$$
\sum_{k=1}^{\infty}\int_0^1f^2(x)\chi_{ \{x :   (k-1)\epsilon \le f(x)<k \epsilon \}}(x)dx \sum_{n=k}^{\infty}\frac{1}{n^2}\le
$$

$$
2 \sum_{k=1}^{\infty} \frac{1}{k}\int_0^1f^2(x)\chi_{ \{x :  (k-1)\epsilon \le f(x)<k \epsilon \}}(x)dx \le
$$

$$
2 \epsilon \sum_{k=1}^{\infty}\int_0^1f(x)\chi_{ \{x :  (k-1)\epsilon \le f(x)<k \epsilon \}}(x)dx = 2 \epsilon \int_0^1f(x)dx.
$$

\end{proof}

We put
$$
F_m((x_i)_{i \in N})=\overline{\lim}_{N \to \infty}\frac{1}{N^2}(\sum_{k=1}^N\big(f(x_k) \chi_{ \{x :  f(x)<\frac{k}{m} \} }(x_k)
-\int_{  \{x :  f(x)<\frac{k}{m} \} }f(x)dx\big))^2.
$$

The next lemma is a simple consequence of Lemma 2.3.

\begin{lem}
   Let $f$ be a Lebesgue integrable  non-negative  real-valued function on $(0,1)$.
  Then the following inequality
    $$
  \int_{(0,1)^{\infty}}F_m((x_i)_{i \in N})d  \ell_1^{\infty}((x_i)_{i \in N}) \le \frac{2\int_0^1f(x)dx}{m}
  $$
 holds true  for $m \in \mathbf{N}$.\end{lem}

\begin{lem}
 Let $f$ be a Lebesgue integrable  non-negative  real-valued function on $(0,1)$. Then  we have $$
\ell_1^{\infty}(\{ (x_k)_{k \in {\bf N}} : (x_k)_{k \in
{\bf N}} \in [0,1]^{\infty} ~\& ~\lim_{s \to
\infty}F_s((x_i)_{i \in N})=0\})=1.
$$
\end{lem}

\begin{proof}
Note that
$$
\sum_{m=1}^{\infty}\int_{(0,1)^{\infty}}F_{m^2}((x_i)_{i \in N})d \lambda^{\infty}((x_i)_{i \in N}) \le
2\int_0^1f(x)dx\sum_{m=1}^{\infty}\frac{1}{m^2}< +\infty.
$$

By Levi well known theorem $(F_{m^2})_{m \in N}$ tends to zero as well $m$ tends to $\infty$ for $\ell_1^{\infty}$-almost everywhere on
$(0,1)^{\infty}$. For $1< s \in \mathbf{N}$ there is $m \ge 1 $ such that $m^2 \le s <(m+1)^2$. Then we get
$$
F_s=\overline{\lim}_{N \to \infty}\frac{1}{N^2}(\sum_{k=1}^N\big(f(x_k) \chi_{  \{x :  f(x)<\frac{k}{s} \} }(x_k)
-\int_{  \{x :  f(x)<\frac{k}{s} \} }f(x)dx))^2
$$
$$
\le \overline{\lim}_{N \to \infty}\frac{1}{N^2}(\sum_{k=1}^N\big(f(x_k) \chi_{  \{x :  f(x)<\frac{k}{m^2} \} }(x_k)
-\int_{  \{x :  f(x)<\frac{k}{(m+1)^2} \} }f(x)dx))^2=
$$
$$
\le \overline{\lim}_{N \to \infty}\frac{1}{N^2}(\sum_{k=1}^N \big(f(x_k) \chi_{  \{x :  f(x)<\frac{k}{m^2} \} }(x_k)
$$
$$
-(\int_{  \{x :  f(x)<\frac{k}{m^2} \} }f(x)dx -\int_{ \{x :  \frac{k}{(m+1)^2} \le f(x)<\frac{k}{m^2} \} }f(x)dx)))^2\le
$$
$$
\le \overline{\lim}_{N \to \infty}\frac{1}{N^2}(\big (\sum_{k=1}^N\big(f(x_k) \chi_{  \{x :  f(x)<\frac{k}{m^2} \} }(x_k)
$$
$$
-\int_{  \{x :  f(x)<\frac{k}{m^2} \} } f(x)dx\big ) + \int_{ \{x :  \frac{k}{(m+1)^2} \le f(x)<\frac{k}{m^2} \} }f(x)dx))^2\le
$$
$$
\le \overline{\lim}_{N \to \infty}\frac{1}{N^2}\big (\sum_{k=1}^N\big(f(x_k) \chi_{ \{x :  f(x)<\frac{k}{m^2} \} }(x_k)
-\int_{ \{x :  f(x)<\frac{k}{m^2} \} }  f(x)dx\big ))^2 +
 $$
 $$
 2\overline{\lim}_{N \to \infty}\frac{1}{N} |\sum_{k=1}^N\big(f(x_k) \chi_{ \{ x : f(x)<\frac{k}{m^2} \} }(x_k)
 $$
 $$
 -\int_{ \{x :  f(x)<\frac{k}{m^2} \} }
 f(x)dx|  \int_{ \{x :  \frac{k}{(m+1)^2} \le f(x)<\frac{k}{m^2} \} }f(x)dx+
 $$
 $$
 \overline{\lim}_{N \to \infty}\frac{1}{N}\big(\int_{ \{x : \frac{k}{(m+1)^2} \le f(x)<\frac{k}{m^2} \} }f(x)dx\big)^2.
$$

Since the right side of  the last equality  tends to  zero when $m$ tends
to $+\infty$(equivalently, $s$ tends to $+\infty$), we end the
proof of Lemma 2.5.

\end{proof}

\begin{lem}
 Let $f$ be a Lebesgue integrable  non-negative  real-valued function on $(0,1)$.  Then the
following equality
    $$
 \lambda^{\infty}( \cup_{n=1}^{\infty}(0,1)^n \times \prod_{k \ge n} \{ x: f(x) < k \epsilon \})=1
  $$
  holds true.
  \end{lem}

\begin{proof} Since $f$ is Lebesgue integrable we have
$$
\sum_{k=1}^{\infty}\ell_1 (\{ x: f(x) \ge   k \epsilon \})=$$
$$
\sum_{k=1}^{\infty}\ell_1(\{ x : f(x) > k \epsilon \})+ \sum_{k=1}^{\infty}\ell_1(\{ x : f(x) = k \epsilon \}) \le  \int_0^1f(x)dx +1.
$$

The last relation means that $\lim_{n \to \infty} \sum_{k=n}^{\infty}\ell_1 (\{ x: f(x) \ge  k \epsilon \})=0$.
Take into account this fact, we get
$$\lim_{n \to \infty}\prod_{k \ge n}\ell_1 (\{ x: f(x) <  k \epsilon \})=\lim_{n \to \infty}\prod_{k \ge n}(1-\ell_1 (\{ x: f(x) \ge  k \epsilon \}))=
$$
$$
e^{\lim_{n \to \infty}\sum_{k \ge n}\ln (1-\ell_1 (\{ x: f(x) \ge  k \epsilon \}))}=e^{\lim_{n \to \infty}\sum_{k \ge n}\ell_1 (\{ x: f(x) \ge  k \epsilon \})}=1.
$$
Hence, for $m \in N $ there is such a natural number $N(m)$  that
$\prod_{k \ge N(m)} \ell_1 (\{ x: f(x) <  k \epsilon \})>1-\frac{1}{m}$.

Now it is obvious that
$$
 \lambda^{\infty}( \cup_{n=1}^{\infty}(0,1)^n \times \prod_{k \ge n} \{ x: f(x) <  k \epsilon \})\ge
 $$
$$
 \lambda^{\infty}( \cup_{m=1}^{\infty}(0,1)^{N(m)-1}\times \prod_{k \ge N(m)}  \{ x: f(x) <  k \epsilon \})=1.
 $$

\end{proof}

\begin{lem}
 Let $f$ be a Lebesgue integrable  non-negative  real-valued function on $(0,1)$.
Then  we have $$
\ell_1^{\infty}(\{ (x_k)_{k \in {\bf N}} : (x_k)_{k \in
{\bf N}} \in [0,1]^{\infty} ~\& ~
\lim_{N \to \infty}\frac{1}{N}\sum_{k=1}^Nf(x_k)=\int_{0}^1f(x)dx
\})=1.
$$
\end{lem}

\begin{proof}  We put $$D=\{ (x_k)_{k \in {\bf N}} : (x_k)_{k \in
{\bf N}} \in [0,1]^{\infty} ~\& ~\lim_{s \to
\infty}F_s((x_i)_{i \in N})=0\}
$$
and

$$E=\cap_{s=1}^{\infty} \cup_{n=1}^{\infty}(0,1)^n \times \prod_{k \ge n}\{ x: f(x) <  \frac{k}{s} \}
$$
By using Lemmas 2.5-2.6, we deduce that $\ell_1^{\infty}(D \cap E)=1$.

Let  $(x_k)_{k \in
{\bf N}} \in  D \cap E$.

Since $(x_k)_{k \in
{\bf N}} \in  D$, we have  $\lim_{s \to
\infty}\sqrt{F_s((x_i)_{i \in N})}=0$.
The latter relation means that for $\epsilon>0$ there is $s_0$ such that
$\sqrt{F_s((x_i)_{i \in N})}<\epsilon$ for $s\ge s_0$, equivalently,
$$
\overline{\lim}_{N \to \infty}\frac{1}{N}|\sum_{k=1}^N\big(f(x_k) \chi_{ \{x : f(x)<\frac{k}{s} \} }(x_k) -\int_{ \{x: f(x)<\frac{k}{s} \} }f(x)dx|<
\epsilon
$$
for $s\ge s_0$.
Since $(x_k)_{k \in
{\bf N}} \in  E$, there is $s_1$ such that $f(x_k)<\frac{k}{s}$ for each $k \ge s_1$.

This means that
$$
\overline{\lim}_{N \to \infty}\frac{1}{N}|\sum_{k=1}^N\big(f(x_k) \chi_{ \{x : f(x)<\frac{k}{s} \} }(x_k) -\int_{ \{x: f(x)<\frac{k}{s} \} }f(x)dx|=
$$
$$
\overline{\lim}_{N \to \infty}\frac{1}{N}|\sum_{k=1}^{s_1}\big(f(x_k) \chi_{ \{x : f(x)<\frac{k}{s} \} }(x_k)
-\int_{ \{x : f(x)<\frac{k}{s} \} }f(x)dx)+
$$
$$
\sum_{k=s_1+1}^N\big(f(x_k) \chi_{ \{x: f(x)<\frac{k}{s} \} }(x_k) -\int_{ \{x : f(x)<\frac{k}{s} \} }f(x)dx)|=
$$
$$
\overline{\lim}_{N \to \infty}\frac{1}{N}|\sum_{k=1}^{s_1}\big(f(x_k) \chi_{ \{x: f(x)<\frac{k}{s} \} }(x_k)
-\int_{ \{x : f(x)<\frac{k}{s} \} }f(x)dx)+$$
$$
+\sum_{k=1}^{s_1}\big(f(x_k) -\int_{ \{x : f(x)<\frac{k}{s} \} }f(x)dx)+
$$
$$
\sum_{k=s_1+1}^N\big(f(x_k) \chi_{ \{x: f(x)<\frac{k}{s} \} }(x_k) -\int_{ \{x : f(x)<\frac{k}{s} \} }f(x)dx)|=
$$
$$
\overline{\lim}_{N \to \infty}\frac{1}{N}|
\sum_{k=1}^{s_1}\big(f(x_k) -\int_{ \{x : f(x)<\frac{k}{s} \} }f(x)dx)+$$
$$\sum_{k=s_1+1}^N\big(f(x_k) \chi_{ \{x: f(x)<\frac{k}{s} \} }(x_k) -\int_{ \{x : f(x)<\frac{k}{s} \} }f(x)dx)|=
$$
$$
\overline{\lim}_{N \to \infty}\frac{1}{N}|
\sum_{k=1}^N\big(f(x_k) -\int_{ \{x : f(x)<\frac{k}{s} \} }f(x)dx)|<\epsilon.
$$
Since $\epsilon$ was taken arbitrary and
$$
\overline{\lim}_{N \to \infty}\frac{1}{N}|
\sum_{k=1}^N\big(f(x_k) -\int_{ \{x : f(x)<\frac{k}{s} \} }f(x)dx)|<\epsilon,
$$
we deduce that
$$
\overline{\lim}_{N \to \infty}\frac{1}{N}|
\sum_{k=1}^N\big(f(x_k) -\int_{ \{x : f(x)<\frac{k}{s} \} }f(x)dx)|=0,
$$
which means that
$$
\lim_{N \to \infty}\frac{1}{N}
\sum_{k=1}^N (f(x_k) -\int_{ \{x : f(x)<\frac{k}{s} \} }f(x)dx)=0.
$$
Since $\lim_{k \to \infty}\int_{ \{x : f(x)<\frac{k}{s} \} }f(x)dx=\int_{0}^1f(x)dx$, by Toeplitz lemma we deduce that
$$
\lim_{N \to \infty}\frac{1}{N}
\sum_{k=1}^N\int_{ \{x : f(x)<\frac{k}{s} \} }f(x)dx)=\int_{0}^1f(x)dx.
$$
The latter relation implies that
$$
\lim_{N \to \infty}\frac{1}{N}\sum_{k=1}^Nf(x_k)=
\lim_{N \to \infty}\frac{1}{N}
\sum_{k=1}^N\int_{ \{x : f(x)<\frac{k}{s} \} }f(x)dx=\int_{0}^1f(x)dx.
$$
This ends the proof of Lemma 2.7.
\end{proof}

\begin{rem} Formulation of Lemma 2.4(cf. \cite{Pant2015}, p.285) needs a certain specification. More precisely, it should be  formulated for  sequences $(x_k)_{k \in N} \in S \cap E \cap D $, where $S$ comes from Lemma 2.1 and,   $E$ and $D$ come from the Lemma 2.7.  Since $\ell_1(S \cap E \cap D)=1$, such reformulated  Lemma 2.4 can be used for the proof of Corollary 4.2(cf. p. 296).
\end{rem}

\section{Main Results}

By using Lemmas  2.1 and 2.7, we get

\begin{thm}
 Let $f$ be a Lebesgue integrable  real-valued function on $(0,1)$.
Then  we have $$
\ell_1^{\infty}(\{ (x_k)_{k \in {\bf N}} : (x_k)_{k \in
{\bf N}} \in [0,1]^{\infty} ~\&
$$
$$(x_k)_{k \in
{\bf N}} ~\mbox{is~uniformly~distributed~in~}(0,1)~\&~
\lim_{N \to \infty}\frac{1}{N}\sum_{k=1}^Nf(x_k)=\int_{0}^1f(x)dx
\})=1.
$$
\end{thm}

\begin{proof} Note that $f=f^{+}+f^{-}$, where
$f^{+}(x)=\sup\{f(x),0\}$ and  $f^{-}(x)=\inf\{f(x),0\}$ for $x \in (0,1)$. Clearly, $f^{+}$ satisfies the condition of Lemma 2.7.
We put
$$
D_1=\big \{ (x_k)_{k \in {\bf N}} : (x_k)_{k \in
{\bf N}} \in [0,1]^{\infty} ~\& ~
\lim_{N \to \infty}\frac{1}{N}\sum_{k=1}^Nf^{+}(x_k)=\int_{0}^1f^{+}(x)dx \big \}.
$$
By Lemma 2.7 we get $\ell_1^{\infty}(D_1)=1$.

Note that $-f^{-}$ also satisfies the condition of Lemma 2.7.
We put
$$
D_2=\big \{ (x_k)_{k \in {\bf N}} : (x_k)_{k \in
{\bf N}} \in [0,1]^{\infty} ~\& ~
\lim_{N \to \infty}\frac{1}{N}\sum_{k=1}^N(-f^{-})(x_k)=-\int_{0}^1(-f^{-})(x)dx \big \}.
$$
By Lemma 2.7 we get $\ell_1^{\infty}(D_2)=1$.

It is obvious that $\ell_1^{\infty}(D_1 \cap D_2 \cap S)=1$, where $S$ comes from Lemma 2.1. For  $(x_k)_{k \in {\bf N}} \in  D_1 \cap D_2 \cap S$ we get
$$
\lim_{N \to \infty}\frac{1}{N}\sum_{k=1}^Nf(x_k)=\lim_{N \to \infty}\frac{1}{N}\sum_{k=1}^N\big( f^{+}-(-f^{-})\big )(x_k)=
$$
$$
\lim_{N \to \infty}\frac{1}{N}\sum_{k=1}^N f^{+}(x_k)- \lim_{N \to \infty}\frac{1}{N}\sum_{k=1}^N(-f^{-})(x_k)=
$$
$$
\int_{0}^1f^{+}(x)dx-\int_{0}^1(-f^{-})(x)dx=\int_{0}^1\big(f^{+}(x)+f^{-}(x)\big )dx=\int_{0}^1f(x)dx.
$$

\end{proof}

Now we present our proof of the  Kolmogorov strong law of large numbers(see, \cite{Shiryaev1980}, Theorem 3(Kolmogorov),
p.
379).

\begin{thm} Let $(\xi_k)_{k \in ~{N}}$ be a sequence of independent identically distributed  random variables with distribution function $F$ for which $M|\xi_1|<\infty.$ Then the condition
 $$
P\{ \omega: \lim_{n \to \infty}\frac{\sum_{k=1}^n\xi_k(\omega)}{n}=m\}=1 \eqno (17)
$$
holds true, where $m=M(\xi_1)$.
\end{thm}

\begin{proof}  Without loss of generality we can assume that
$$
(\Omega,\mathcal{F},P)=((0,1)^{\infty}, \mathcal{B}((0,1)^{\infty}),\ell_1^{\infty}). \eqno (4.9)
$$

Let $(c_k)_{k \in N}$ be a set of all points of jumps of $F$. We denote by $d_k$ the jump of a function $F$ at point $c_k$ for each $k \in N$.

Let $f:(0,1) \to R$ be
defined by
$$
f(x)= \sum_{k \in N} c_kI_{[F(c_k)-d_k,F(c_k)[}(x) + \sup \{ y : F(y)=x \} I_{(0,1) \setminus \cup_{k \in N} [F(c_k)-d_k,F(c_k)[}(x)\eqno (4.11)
$$
for each $x \in (0,1)$.

Note that $f$ is a Lebesgue integrable real-valued function on $(0,1)$ such that
  $$
 \int_{0}^1f(x)dx=\int_{-\infty}^{+\infty}x d F(x)=m.
$$
An application of Theorem 3.1 ends the proof of Theorem 3.2.
\end{proof}
\begin{rem}
Let $f:(0,1)\to {\bf R}$ be a Lebesgue integrable function.  By Theorem 3.1 we have
$\ell_1^{\infty}(A_f)=1$, where
$$
A_f=\{ (x_k)_{k \in {\bf N}} : (x_k)_{k \in
{\bf N}} \in (0,1)^{\infty} ~\& ~\lim_{N \to
\infty}\frac{1}{N}\sum_{n=1}^Nf(x_n)=\int_{(0,1)}
f(x)dx\}.
$$
\end{rem}
We have the following simple consequences of Theorem 3.1.

\begin{cor}Let $f:(0,1)\to {\bf R}$ be Lebesgue integrable function.  Then   we have
$\ell_1^{\infty}(B_f)=1$, where
$$
B_f=\{ (x_k)_{k \in {\bf N}} : (x_k)_{k \in
{\bf N}} \in (0,1)^{\infty} ~\& ~\lim_{N \to
\infty}\frac{1}{N}\sum_{k=1}^Nf(x_k)~\mbox{exists} \}.
$$\end{cor}
\begin{proof}  Since $A_f \subseteq B_f$, by Remark 3.1 we get
$$
1 =\ell_1(A_f)\le \ell_1(B_f)\le \ell_1((0,1)^{\infty})=1.
$$

\end{proof}

\begin{cor} Let $f:(0,1)\to {\bf R}$ be Lebesgue integrable function.  Then   we have
$\ell_1^{\infty}(C_f)=1$, where
$$
C_f=\{ (x_k)_{k \in {\bf N}} : (x_k)_{k \in
{\bf N}} \in (0,1)^{\infty} ~\& ~\lim_{N \to
\infty}\frac{f(x_N)}{N}=0\}.
$$
\end{cor}

\begin{proof} Note that $A_f \subseteq C_f$. Indeed, let $(x_k)_{k \in {\bf N}} \in A_f $. Then we get
$$
\lim_{N \to \infty}\frac{f(x_N)}{N}= \lim_{N \to \infty}\frac{1}{N}(\sum_{k=1}^Nf(x_k)-\sum_{k=1}^{N-1}f(x_k))=
$$
$$
\lim_{N \to \infty}\frac{1}{N}\sum_{k=1}^Nf(x_k)-  \lim_{N \to \infty}\frac{1}{N}\sum_{k=1}^{N-1}f(x_k)=
$$
$$
\lim_{N \to \infty}\frac{1}{N}\sum_{k=1}^Nf(x_k)- \lim_{N-1 \to \infty}\frac{N-1}{N}(\frac{1}{N-1}\sum_{k=1}^{N-1}f(x_k))=
$$
$$
\lim_{N \to \infty}\frac{1}{N}\sum_{k=1}^Nf(x_k)- \lim_{N-1 \to \infty}\frac{1}{N-1}\sum_{k=1}^{N-1}f(x_k)=0.
$$
By Remark 3.1  we know that $\ell_1^{\infty}(A_f)=1$ which implies $1=\ell_1^{\infty}(A_f) \le \ell_1^{\infty}(C_f) \le \ell_1^{\infty}((0,1)^{\infty})=1.$

\end{proof}

\begin{rem}Note that for each Lebesgue integrable function $f$ in $(0,1)$, the following  inclusion $S \cap A_f \subseteq S \cap C_f$ holds true, but the converse inclusion is not always valid. Indeed, let $(x_k)_{k \in N}$ be an arbitrary sequence of uniformly
distributed numbers in $(0,1)$. Then the function $f:(0,1)\to \mathbf{R}$,  defined by $f(x)= \chi_{(0,1)\setminus \{x_k:k \in {\bf N}\}}(x)$ for $x \in (0,1)$, is Lebesgue integrable, $(x_k)_{k \in N} \in C_f \cap S$ but  $(x_k)_{k \in N} \notin A_f \cap S$ because
$$
\lim_{N \to
\infty}\frac{1}{N}\sum_{n=1}^Nf(x_n)=0 \neq 1 = \int_{(0,1)}
f(x)dx.
$$
\end{rem}

\begin{thm}  Let $f:(0,1)\to {\bf R}$ be Lebesgue integrable function.
Then the set $S_f$ of all  sequences $(x_k)_{k \in {\bf N}} \in (0,1)^{\infty}$  for which the following conditions

1) $\lim_{n \to \infty}\frac{f(x_n)}{n}=0$;

2) ~$\lim_{N \to \infty}\frac{1}{N}\sum_{k=1}^Nf(x_k) $ exists;

3) $\lim_{N \to \infty}\frac{1}{N}\sum_{k=1}^Nf(x_k) = \int_{(0,1)}f(x)dx$;

4) ~ $(x_k)_{k \in {\bf N}}$ is uniformly distributed in $(0,1)$

are equivalent,  has  $\ell_1^{\infty}$ measure one.
\end{thm}

\begin{proof} By Lemma 2.1 we know that $\ell_1^{\infty}(S)=1$.  By Remark 3.1 we have $\ell_1^{\infty}(A_f)=1$.
Following Corollaries  3.1 and 3.2 we have $\ell_1^{\infty}(B_f)=1$
and $\ell_1^{\infty}(C_f)=1$, respectively.  Since $S_f=A_f \cap B_f \cap C_f \cap S$, we get
$$
\ell_1^{\infty}(D)=\ell_1^{\infty}(A_f \cap B_f \cap C_f \cap S)=1.
$$

\end{proof}

The next corollary is a simple consequence of Theorem 3.3.

\begin{cor}  Let  $\mathbf{Q}$ be a set of all rational numbers of  $[0,1]$  and $F \subseteq [0,1]\cap \mathbf{Q}$ be finite. Let
$f : [0,1] \to R$ be Lebesgue integrable, $\ell_1$-almost everywhere
continuous and locally bounded on $[0,1] \setminus F$. Assume that for every $\beta \in F$  there is some neighbourhood $U_{\beta}$ of $\beta$
such that   $f$ is either bounded or monotone in $[0,\beta) \cap U_{\beta}$ and in $(\beta,1]\cap U_{\beta}$ as well.
Let $S,A_f,B_f,C_f$ come from Lemma 2.1, Remark 3.1, Corollary 3.1,Corollary 3.2, respectively.
We set
$$S_f=\big(A_f \cap B_f \cap C_f \cap S) \cup ((0,1)^{\infty} \setminus A_f) \cap \big((0,1)^{\infty} \setminus B_f) \cap ((0,1)^{\infty} \setminus C_f)\cap S).
$$

Then for $(x_k)_{k \in {\bf N}} \in S_f$  the following  conditions are equivalent:

1) $\lim_{n \to \infty}\frac{f(x_n)}{n}=0$;

2) ~$\lim_{N \to \infty}\frac{1}{N}\sum_{k=1}^Nf(x_k) $ exists;

3) $\lim_{N \to \infty}\frac{1}{N}\sum_{k=1}^Nf(x_k) = \int_{(0,1)}f(x)dx$;

\end{cor}

\begin{rem} Note that $S_f$ is maximal subset of the set $S$ for which conditions 1)-3) of Corollary 3.3  are equivalent, provided that
for each $(x_k)_{k \in N} \in  S_f$ the sentences 1)-3) or their negations are  true simultaneously, and
for each $(x_k)_{k \in N} \in S \setminus S_f$ the sentences 1)-3) or their negations are not true simultaneously. This extends  the main result of Baxa and Schoi$\beta$engeier \cite{BaxSch2002} because, the sequence of the form
$(\{n\alpha\})_{n \in N}$  is  in $S_f$  for each irrational number $\alpha$, and no every element of $S_f$ can be presented in the same form.
For example,
$$(\{(n+1/2(1-\chi_{\{ k: k\ge 2\}}(n)))\pi^{\chi_{\{ k: k\ge 2\}}(n)}\})_{n \in {\bf N}} \in D \setminus S^{*},$$
where $\{\cdot\}$ denotes the fractional part of the real number and  $\chi_{\{ k: k\ge 2\}}$ denotes the indicator function of the set $\{ k: k\ge 2\}$.

Similarly, setting
$$D_f=\big(A_f \cap B_f \cap C_f) \cup ((0,1)^{\infty} \setminus A_f)\big) \cap \big((0,1)^{\infty} \setminus B_f) \cap ((0,1)^{\infty} \setminus C_f)\big),
$$
we get a maximal subset of $(0,1)^{\infty}$ for which conditions 1)-3) of Corollary 3.3 are equivalent, provided that
for each $(x_k)_{k \in N} \in  D_f$ the sentences 1)-3) or their negations are  true simultaneously, and for each $(x_k)_{k \in N} \in
(0,1)^{\infty} \setminus D_f$ the sentences 1)-3) or their negations are not true simultaneously.
\end{rem}


\end{document}